\documentclass[11pt]{myarticle}    
\author{Benjamin J. Wilson}
\title{A Character Formula for the Category $\categoryo$}
\usepackage{graphicx}  
\usepackage{subfigure}
\usepackage{ifthen}
\usepackage{geometry}
\usepackage[center]{caption} 

\theoremstyle{definition}

\newtheorem*{theoremnonum}{Theorem}

\newtheorem{theorem}[equation]{Theorem}
\newtheorem{propn}[equation]{Proposition}
\newtheorem{lemma}[equation]{Lemma}
\newtheorem{corollary}[equation]{Corollary}
\newtheorem{remark}[equation]{Remark}

\newcommand{\order}[1]{\text{ord}\hspace{0.1em}{#1}}
\newcommand{\irredcycle}{\upeta}
\newcommand{\loopend}{\hat{\upeta}}
\newcommand{\cycle}[1]{\upsigma_{#1}}
\newcommand{\vscomponent}[2]{#1(#2)}
\newcommand{\vscharacter}[1]{\mathscr{P}_{#1}}
\newcommand{\basis}{\mathrm{B}}
\newcommand{\params}[2]{\mathrm{D}_{#1, #2}}
\newcommand{\paramcycle}[1]{\uptau_{#1}}
\newcommand{\ordercount}[3]{\mathrm{O}_{#1, #2}(#3)}
\newcommand{\ordergf}[2]{\mathscr{O}_{#1, #2}}
\newcommand{\annihilator}[1]{\text{ann}(#1)}
\newcommand{\rous}[1]{\Re (#1)}
\newcommand{\rou}[1]{\zeta_{#1}}
\newcommand{\eigenspace}[3]{{#1}|^{#2}_{#3}}
\newcommand{\kroneckersymbol}{\delta}
\newcommand{\epbasis}[2]{\uptheta_{#1, #2}}
\newcommand{\onedimvector}[1]{\text{u}_{#1}}
\newcommand{\onedim}[1]{\K \onedimvector #1}

\newcommand{\K}{\Bbbk}

\newcommand{\Z}{\mathbb{Z}}
\newcommand{\Zmod}[1]{\Z_{#1}}
\newcommand{\N}{\mathbb{N}}
\newcommand{\zplus}{\Z_{+}}
\newcommand{\kcross}{\K^{\times}}
\newcommand{\epiarrow}{\twoheadrightarrow}
\newcommand{\kronecker}[2]{\delta_{#1, #2}}
\newcommand{\divides}[2]{#1  \mid #2}

\newcommand{\lineardual}[1]{#1^{\ast}} 
\newcommand{\indet}[1]{\mathrm{#1}} 
\newcommand{\indett}{\indet{t}} 
\newcommand{\laurent}{\K [ \indett, \indett^{-1}]}
\newcommand{\sublaurent}[1]{\K [ \indett^{#1}, \indett^{-#1}]}

\newcommand{\Span}{\text{span}}

\newcommand{\polyring}[2]{#1 [ \indet{#2}]}
\newcommand{\laurentabbrev}{\mathcal{A}}
\newcommand{\h}{\mathfrak{h}}
\newcommand{\g}{\mathfrak{g}}
\newcommand{\ghat}{\hat{\g}}
\newcommand{\affineg}{\tilde{\g}}
\newcommand{\affinecsa}{\tilde{\h}}
\newcommand{\hhat}{\hat{\h}}
\newcommand{\UEA}[1]{\mathcal{U}(#1)}
\newcommand{\SymAlg}[1]{\text{S}(#1)}
\newcommand{\gd}{\mathrm{d}}
\newcommand{\sle}{\mathrm{e}}
\newcommand{\slf}{\mathrm{f}}
\newcommand{\slh}{\mathrm{h}}
\newcommand{\SL}[1]{\mathrm{sl}(#1)}
\newcommand{\LieBrac}[2]{\lbrack \hspace{0.15em} #1 , #2 \hspace{0.15em}\rbrack} 
\newcommand{\groot}{\upalpha}
\newcommand{\imagroot}{\updelta}
\newcommand{\allfunctions}{{\mathcal F}} 
\newcommand{\exppoly}{{\mathcal E}} 
\newcommand{\charpoly}[1]{\mathrm{c}_{#1}}

\newcommand{\categoryo}{\tilde{\mathcal{O}}}
\newcommand{\hmod}[1]{\mathbf{H}(#1)}
\newcommand{\affineirred}[1]{\mathbf{N}(#1)}

\newcommand{\expmap}[1]{\textsc{exp}(#1)}
\newcommand{\loopmod}[1]{{\widehat{#1}}}

\newcommand{\ind}[1]{\mathfrak{V}(#1)}
\newcommand{\irred}[1]{\mathfrak{L}(#1)}

\newcommand{\weightspace}[2]{{#1}_{#2}}
\newcommand{\support}[1]{{\mathrm{s}}(#1)}
\newcommand{\Ind}[3]{\text{Ind}_{#1}^{#2} \  #3}
\newcommand{\Nilp}{\mathrm{N}} 
\newcommand{\gen}[1]{\textrm{v}_{#1}}
\newcommand{\epquotient}[2]{#1({#2})}
\newcommand{\kronmod}[1]{\wp_{#1}}

\newcommand{\character}[1]{\mathrm{char}\hspace{0.15em} {#1}}
\newcommand{\exppolychar}[1]{ \character{\affineirred{#1}}}
\newcommand{\ramanujan}[2]{\mathrm{c}_{#1}(#2)}
\newcommand{\eulertotient}{\upphi}
\newcommand{\moebiusmu}{\upmu}
\newcommand{\fouriercoeff}[1]{{\mathrm P}_{#1}}

\newcommand{\heis}{\mathfrak{t}}

\newcommand{\ts}{\textstyle}
\newcommand{\ie}{i.e.\ }

\newcommand{\EditingNote}[1]{}
\newcommand{\newterm}[1]{\textit{#1}}
\newcommand{\Note}[1]{\text{\small{\ (#1)}}}
\newcommand{\ENote}[1]{\text{\small{\quad (#1)}}}

\newcommand{\SFrac}[2]{\textstyle{\frac{#1}{#2}}}

\begin{document}
\begin{abstract}
One may construct, for any function on the integers, an irreducible module of level zero for affine $\SL 2$, using the values of the function as structure constants.  
The modules constructed using exponential-polynomial functions realise the irreducible modules  with finite-dimensional weight spaces in the category $\categoryo$ of Chari.
In this work, an expression for the formal character of such a module is derived using the highest-weight theory of truncations of the loop algebra.
\end{abstract}

\maketitle

\begin{section}{Introduction}\label{introduction}
Let $\g$ denote the Lie algebra $\SL 2$ of $2 \times 2$ traceless matrices over an algebraically closed field $\K$ of characteristic zero.
Associated to $\g$ is the centreless affine Lie algebra $\affineg$,
$$\affineg = \g \otimes \K [ \indett, \indett^{-1}] \oplus \K \gd$$
obtained as an extension of the loop algebra by a degree derivation $\gd$.
The Cartan subalgebra of $\affineg$ is given by $\affinecsa = \h \oplus \K \gd$, where $\h$ is the Cartan subalgebra of $\g$.
The category $\categoryo$, introduced by Chari \cite{Chari86}, is an analogue of the category $\mathcal O$ of Bernstein, Gelfand and Gelfand \cite{BGG}, corresponding to the natural Borel subalgebra of $\affineg$ (that is, to the loop algebra of the Borel subalgebra of $\g$).
For every exponential-polynomial function $\varphi : \Z \to \K$, one may construct an irreducible $\affineg$-module $\affineirred \varphi$ in the category $\categoryo$ with finite-dimensional weight spaces.
It follows from Chari \cite{Chari86}, and from Billig and Zhao \cite{BilligZhao}, that any irreducible module in the category $\categoryo$ with all weight spaces of finite dimension can be obtained as the tensor product of a one-dimensional module with a module $\affineirred \varphi$.
In this paper, a formula for the character of the modules $\affineirred \varphi$ is derived, thus describing the character of the irreducible modules in the category $\categoryo$ with finite-dimensional weight spaces.

For any weight $\affineg$-module $M$, write
$$ M = \bigoplus_{\chi \in \lineardual \affinecsa} \weightspace M \chi \quad \text{where} \quad h |_{\weightspace M \chi} = \chi (h), \ \ h \in \affinecsa,$$
and if the weight spaces $\weightspace M \chi$ are finite dimensional, write
$$ \character M = \sum_{\chi \in \lineardual \affinecsa} \dim \weightspace M \chi \hspace{0.15em} \indet{Z}^\chi$$
for its character.

A function $\varphi: \Z \to \K$ is \newterm{exponential polynomial} if it can be written as a finite sum of products of polynomial and exponential functions.
For any $\lambda \in \K^{\times}$, define the exponential function\newnot{notn:expmap}
$$ \expmap \lambda : \Z \to \K, \qquad \expmap \lambda (m) = \lambda^m, \quad m \in \Z.$$
For any positive integer $r$, define the function\newnot{notn:kronmod}
$$\kronmod r = \sum_{\zeta^r = 1} \expmap{\zeta},$$
where the sum is over all roots of unity $\zeta$ such that $\zeta^r = 1$.
If $\varphi: \Z \to \K $ is a non-zero exponential-polynomial function, then there exists a unique $r >0$ such that  
\begin{equation}\label{exppolyexpression}
\varphi = \kronmod r \cdot \sum_i a_i \hspace{0.1em} \expmap{\lambda_i}
\end{equation}
for some finite collection of polynomial functions $a_i : \Z \to \K $ and scalars $\lambda_i \in \K^{\times}$, such that if 
$(\lambda_i / \lambda_j)^r = 1 $, then $i = j$. 
Let $\groot$ denote the positive root of $\g$, and let $\imagroot$ denote the fundamental imaginary root of $\affineg$.
The characters of the modules $\affineirred \varphi$ are described by the following theorem.
\begin{theoremnonum}
Let $\varphi : \Z \to \K $ be a non-zero exponential-polynomial function.
In the notation of (\ref{exppolyexpression}),
\begin{equation}\label{characterformula}
\exppolychar \varphi = \indet{Z}^{\varphi (0) \frac{\groot}{2}} \cdot \frac{1}{r} \sum_{n \in \Z} \sum_{\divides d r} \ramanujan d n \left ( \fouriercoeff \varphi (\indet{Z}^{-d\groot}) \right ) ^{\frac r d} \indet{Z}^{n\imagroot},
\end{equation}
where the inner sum is over the positive divisors $d$ of $r$, the quantities $\ramanujan d n$ are Ramanujan sums, and 
$$
\fouriercoeff \varphi (\indet Z) = \frac{ \prod_{a_i \in \zplus} (1 - \indet{Z}^{a_i + 1})} { (1 - \indet Z)^M},
$$
where $M = \sum_i (\deg {a_i} + 1)$ and the product is over those indices $i$ for which $a_i \in \zplus$.
\end{theoremnonum}
The \newterm{Ramanujan sum} $\ramanujan d n$ is given by
\begin{equation}\label{ramanujandefn}
\ramanujan d n =  \frac{ \eulertotient(d) \moebiusmu(d')}{\eulertotient(d')}, \qquad d' = \frac{d}{\gcd (d,n)},
\end{equation}
where $\eulertotient$ denotes Euler's totient function and $\moebiusmu$ denotes the M\"obius function.

It may be deduced from the character formula \eqref{characterformula} that if the weight space $\weightspace {\affineirred \varphi} \chi$ is non-zero, then
$$ \chi =\varphi (0)\frac{\groot}{2} - k\groot + n \imagroot,$$
for some $n \in \Z$ and $k \geqslant 0$.
The function $\ramanujan d {\cdot}$ has period $d$, and so for any $k \geqslant 0$, the multiplicity function
$$ n \mapsto \dim \weightspace {\affineirred \varphi} {\varphi (0)\frac{\groot}{2} - k\groot + n \imagroot} , \qquad n \in \Z,$$
has period $r$.
Therefore the character of $\affineirred \varphi$ is completely described by the array of weight-space multiplicities
$$ \left [ \dim \weightspace {\affineirred \varphi} {\varphi (0)\frac{\groot}{2} - k\groot + n \imagroot} \right ]$$
where $k \geqslant 0$ and $0 \leqslant n < r$.
Examples of these arrays, such as those illustrated by Figures \ref{multexample1} -- \ref{multexample4}, may be computed in a straightforward manner using the formula \eqref{characterformula}.
Columns are indexed left to right by $n$, where $0 \leqslant n < r$, while rows are indexed from top to bottom by $k \geqslant 0$.
\begin{figure}
\begin{center}
\subfigure[$\varphi = \kronmod 6$]{\label{multexample1}
\begin{minipage}[t]{6.5cm}
$$
\begin{array}{cccccc}
1&0&0&0&0&0 \\
1&1&1&1&1&1\\
3&2&3&2&3&2\\
4&3&3&4&3&3\\
3&2&3&2&3&2\\
1&1&1&1&1&1\\
1&0&0&0&0&0\\
0&0&0&0&0&0\\
 &\vdots&&&\vdots&
\end{array}
$$
\end{minipage}}
\subfigure[$ \varphi = - \kronmod 4 (\expmap \lambda + \expmap \mu)$]{\label{multexample2}
\begin{minipage}[t]{6.5cm}
$$
\begin{array}{cccc}
  1&0&0&0 \\
  2&2&2&2 \\
  10&8&10&8 \\
  30&30&30&30 \\
  86&80&84&80 \\
  198&198&198&198 \\
  434&424&434&424 \\
  858&858&858&858 \\
&\vdots&\vdots&
\end{array}
$$
\end{minipage}}
\subfigure[$\varphi = -\kronmod 6$]{\label{multexample3}
\begin{minipage}[b]{6.5cm}
$$
\begin{array}{cccccc}
  1&0&0&0&0&0 \\
  1&1&1&1&1&1 \\
  4&3&4&3&4&3 \\
  10&9&9&10&9&9 \\
  22&20&22&20&22&20 \\
  42&42&42&42&42&42 \\
  80&75&78&76&78&75 \\
  132&132&132&132&132&132 \\
  217&212&217&212&217&212 \\
 335& 333& 333& 335& 333& 333 \\
&\vdots&&&\vdots&
\end{array}
$$
\end{minipage}}
\subfigure[$\varphi = \kronmod 2 (\expmap \lambda - \expmap \mu)$]{\label{multexample4}
\begin{minipage}[b]{6.5cm}
$$
\begin{array}{cc}
  1&0 \\
  2&2 \\
  5&3 \\
  6&6 \\
  9&7 \\
  10&10 \\
  13&11 \\
  14&14 \\
  17&15 \\
  18&18 \\
  \vdots&\vdots
\end{array}
$$
\end{minipage}}
\end{center}
\caption{Array of weight-space multiplicites of $\affineirred \varphi$}
\label{besteira}
\end{figure}

Greenstein \cite{GreensteinCharactersSL} (see also \cite[Section 4.1]{QuantumLoopModules}) has derived an explicit formula for the character of an integrable irreducible object of the category $\categoryo$.
These objects are precisely the exponential-polynomial modules $\affineirred \varphi$ where $\varphi$ is a linear combination of exponential functions with non-negative integral coefficients.
Indeed, our result may alternatively be deduced by considering separately the case where $\affineirred \varphi$ is integrable, employing the result of Greenstein, and the case where $\affineirred \varphi$ is not integrable, using Molien's Theorem.
Our approach, via a general study of finite cyclic-group actions, has the advantage of permitting a unified proof.
Both approaches employ the explicit expression of the character of an irreducible highest-weight module for a truncated current Lie algebra described in \cite{WilsonTCLA}.
\end{section}

\begin{section}{Exponential-Polynomial Functions}
A function $\varphi: \Z \to \K$ is \newterm{exponential polynomial} if it can be written as a finite sum of products of polynomial and exponential functions, \ie 
\begin{equation*}
\sum_{\lambda \in \kcross} \varphi_\lambda \hspace{0.1em} \expmap{\lambda},
\end{equation*}
for some polynomial functions $\varphi_\lambda : \Z \to \K$ and distinct scalars $\lambda \in \kcross$.\newnot{notn:expmapcoeff1}
Write\newnot{notn:exppoly1}
$$ \exppoly = \set{ \varphi : \Z \to \K | \varphi \  \text{is exponential polynomial}}.$$

Let $ \laurentabbrev = \laurent$.

\begin{subsection}{Module structure}
Let $\allfunctions = \set{\varphi: \Z \to \K}$.
Define an endomorphism $\tau$ of the vector space $\allfunctions$ via
$$(\tau \cdot \varphi) (m) = \varphi (m + 1), \qquad m \in \Z, \quad \varphi \in \allfunctions.$$
The rule $\indett \mapsto \tau$ endows $\allfunctions$ with the structure of an $\laurentabbrev$-module.
For any $k \geqslant 0$ and $\lambda \in \kcross$, define the function \newnot{notn:epbasis} $\epbasis \lambda k \in \allfunctions$ by 
$$ \epbasis \lambda k (m) = m^k \lambda^m, \qquad m \in \Z.$$
By definition, $\exppoly$ is the linear subspace of $\allfunctions$ spanned by the $\epbasis \lambda k$.
It is immediate from the following lemma that $\exppoly$ is an $\laurentabbrev$-submodule of $\allfunctions$.

\begin{lemma}\label{annihilationlemma}
For any $\lambda, \mu \in \kcross$ and $k \geqslant 0$, 
\begin{enumerate}
\item\label{annihilationlemma1} $(\indett - \mu) \cdot \epbasis \lambda k = (\lambda - \mu) \epbasis \lambda k + \lambda \sum_{j=0}^{k-1} \binom {k} j \epbasis \lambda {j}$;
\item\label{annihilationlemma2} $(\indett - \lambda)^k \cdot \epbasis \lambda k = k! \lambda^k \epbasis \lambda 0$.
\item\label{annihilationlemma3} $(\indett - \lambda)^{k+1} \cdot \epbasis \lambda k = 0$.
\end{enumerate}
\end{lemma}
\begin{proof}
For any $m \in \Z$,
\begin{eqnarray*}
(\indett \cdot \epbasis \lambda k) (m) = \epbasis \lambda k (m+1) 
	&=& (m+1)^k \lambda^{m+1} \\
	&=& \textstyle \lambda \sum_{j=0}^k \binom k j m^j \lambda^m \\
	&=& \textstyle \lambda \sum_{j=0}^k \binom k j \epbasis \lambda j (m).
\end{eqnarray*}
Therefore,
$$ \textstyle (\indett - \mu) \cdot \epbasis \lambda k = \lambda \sum_{j=0}^k \binom k j \epbasis \lambda j - \mu \epbasis \lambda k,$$
and so part \ref{annihilationlemma1} is proven.
Part \ref{annihilationlemma2} is proven by induction.  The claim is trivial if $k=0$, so suppose that the claim holds for some $k \geqslant 0$.
Then
\begin{eqnarray*}
(\indett - \lambda)^{k+1} \cdot \epbasis \lambda {k+1} &=& \textstyle (\indett - \lambda)^k \cdot \lambda \sum_{j=0}^k \binom {k+1} j \epbasis \lambda j \\
&=& \textstyle \lambda \binom {k+1} k k! \lambda^k \epbasis \lambda 0 \ENote{by inductive hypothesis} \\
&=& (k+1)! \lambda^{k+1} \epbasis \lambda 0,
\end{eqnarray*}
where part \ref{annihilationlemma1} is used in obtaining the first and second equalities.
Therefore the claim holds for all $k \geqslant 0$ by induction.
Part \ref{annihilationlemma3} follows immediately from parts \ref{annihilationlemma1} and \ref{annihilationlemma2}.
\end{proof}

\begin{propn}\label{basispropn}
The set $\set{ \epbasis \lambda k | \lambda \in \kcross, \ k \geqslant 0}$ is linearly independent, and hence is a basis for $\exppoly$.
\end{propn}
\begin{proof}
Suppose that $\gamma_{\lambda, k} \in \K$, $\lambda \in \kcross$, $k \geqslant 0$, are scalars such that the sum
$$ \varphi = \sum_{\lambda \in \kcross} \sum_{k \geqslant 0} \gamma_{\lambda, k} \epbasis \lambda k,$$
is finite and equal to zero.
Write $Z = \set{\lambda \in \kcross | \gamma_{\lambda, k} \ne 0 \ \ \text{for some} \ \ k \geqslant 0}$, and let
$$ n_\lambda = \text{max} \set{ k | \gamma_{\lambda, k} \ne 0}, \quad \lambda \in Z.$$
Then, for any $\lambda \in Z$, 
\begin{eqnarray*}
0 &=& \prod_{\mu \in Z} (\indett - \mu)^{n_\mu + 1 - \kronecker \lambda \mu} \cdot \varphi \\
&=& \prod_{\mu \in Z} (\indett - \mu)^{n_\mu + 1 - \kronecker \lambda \mu} \cdot ( \gamma_{\lambda, n_\lambda} \epbasis \lambda {n_\lambda}) \\
&=& \gamma_{\lambda, n_\lambda} \lambda^{n_\lambda}   {n_\lambda}! \prod_{\mu \in Z, \mu \ne \lambda} (\lambda - \mu)^{n_\mu + 1} \cdot \epbasis \lambda 0, 
\end{eqnarray*}
by Lemma \ref{annihilationlemma}. 
Therefore $\gamma_{\lambda, n_\lambda} = 0$ for all $\lambda \in Z$, which is absurd, unless $Z$ is the empty set.
\end{proof}
\end{subsection}

\begin{subsection}{Recurrence relations}\label{recurrencerelations}
Suppose that $c (\indett) \in \polyring \K t$ is a non-zero polynomial of degree $q$, and write $c (\indett) = \sum_{k = 0}^q c_k \indett^k$.  Then, for any $\varphi \in \allfunctions$, we have $c \cdot \varphi = 0$ if and only if
\begin{equation}\label{recurrencerelation}
0 = (c \cdot \varphi) (m) = c_0 \varphi (m) + c_1 \varphi (m+1) + \cdots + c_q \varphi(m+q), 
\end{equation}
for all $m \in \Z$.
That is, $c \cdot \varphi = 0$ precisely when the values of $\varphi$ satisfy the linear homogeneous recurrence relation with constant coefficients \eqref{recurrencerelation} defined by $c$.
The following proposition characterises the exponential-polynomial functions as the solutions of such recurrence relations.

\begin{propn}\label{exppolychar}
The $\laurentabbrev$-submodule $\exppoly \subset \allfunctions$ is characterised by
$$\exppoly = \set{ \varphi \in \allfunctions | c \cdot \varphi = 0 \ \ \text{for some} \ \ c \in \polyring \K t}.$$
\end{propn}
\begin{proof}
Suppose that $c \in \polyring \K t$ is of degree $q$.
The equation $c \cdot \varphi = 0$ is equivalent to the relation \eqref{recurrencerelation},
 and so the space consisting of all solutions $\varphi$ is at most $q$-dimensional.
Now write $Z \subset \kcross$ for the set of all roots of $c$.
The field $\K$ is algebraically closed, and so
\begin{equation}\label{cfactorization}
c (\indett) \sim_{\kcross} \prod_{\lambda \in Z} (\indett - \lambda)^{m_\lambda },
\end{equation}
where $m_\lambda$ is the multiplicity of the root $\lambda \in Z$.
Lemma \ref{annihilationlemma} shows that the set $$\set{ \epbasis \lambda k | \lambda \in Z, \ \ 0 \leqslant k < m_\lambda},$$ which is of size $\sum_{\lambda \in Z} m_\lambda = q$, consists of solutions to $c \cdot \varphi = 0$. 
By Proposition \ref{basispropn}, this set is linearly independent, and hence is a basis for the solution space.
This proves the inclusion $\supset$.
The inclusion $\subset$ follows immediately from Lemma \ref{annihilationlemma} part \ref{annihilationlemma3}.
\end{proof}
\end{subsection}

\begin{subsection}{Characteristic polynomials}
By Proposition \ref{exppolychar}, for any $\varphi \in \exppoly$, the annihilator $\annihilator \varphi \subset \polyring \K t$ is a non-zero ideal of $\polyring \K t$. 
The unique monic generator \newnot{notn:charpoly2}$\charpoly \varphi \in \annihilator \varphi$ is called the \newterm{characteristic polynomial} of $\varphi$. 

\begin{propn}\label{charpolypropn}
Suppose that $\varphi \in \exppoly$ and write
\begin{equation}\label{charpolypropn1}
\varphi = \sum_{\lambda \in \kcross} \varphi_\lambda \expmap \lambda,
\end{equation}
as a finite sum of products of polynomials functions $\varphi_\lambda$\newnot{notn:expmapcoeff2} and exponential functions $\expmap \lambda$, $\lambda \in \kcross$.
Then
\begin{equation}\label{charpolypropn2}
\charpoly \varphi (\indett) = \prod_{\lambda \in Z} (\indett - \lambda)^{\deg {\varphi_\lambda} + 1},
\end{equation}
where $Z = \set{ \lambda \in \kcross | \varphi_\lambda \ne 0}$.
\end{propn}
\begin{proof}
Suppose that $\varphi \in \exppoly$ has the form (\ref{charpolypropn1}), let $c \in \polyring \K t$ be non-zero, and write $c$ in the form (\ref{cfactorization}).
By Lemma \ref{annihilationlemma}, we have $c \cdot \varphi= 0$ if and only if $m_\lambda > \deg \varphi_\lambda$ whenever $\varphi_\lambda \ne 0$.
The polynomial (\ref{charpolypropn2}) is the minimal degree monic polynomial that satisfies this condition, and hence is the characteristic polynomial.
\end{proof}
\end{subsection}

\begin{subsection}{Degree}
It follows from the characterisation of $\exppoly$ given in subsection \ref{recurrencerelations} that if $\varphi \in \exppoly$ is non-zero,
then the support of $\varphi$ is not wholly contained in either of the infinite subsets of consecutive integers $\N, -\N \subset \Z$.
Lemma \ref{submonoidlemma} demonstrates that the monoid generated by the support of $\varphi$ is of the form $r \Z$, for some unique positive integer $r$.
Write \newnot{notn:functiondegree}$\deg \varphi = r$ for the \newterm{degree} of $\varphi$.

\begin{lemma}\label{submonoidlemma}
Suppose that $A$ is a submonoid of $\Z$ such that $\N, -\N \not \subset A$.
Then $A = r \Z$, where $r \in A$ is any non-zero element of minimal absolute value.
\end{lemma}
\begin{proof}
Let $r \in A \cap \N$ be of minimal absolute value.
For any $m \in A \cap {-\N}$, we have that $m + k r \in A$ where $k$ is the unique positive integer such that 
$$ 0 \leqslant m + k_0 r < r.$$
Thus $m + k r = 0$ by the minimality of $r$; it follows that $r$ divides $m$, for any $m \in A \cap {-\N}$.
Moreover,
$$ -r = m + (k -1) r \quad \in A$$
since $k - 1$ is non-negative.
It follows therefore that $-r$ is the element of minimal absolute value in $A \cap - \N$.
The argument above with inequalities reversed shows that $-r$ divides all positive elements of $A$, and so $A \subset r \Z$.
The opposite inclusion is obvious since $r, -r \in A$ and $A$ is closed under addition.
\end{proof}
\end{subsection}

\begin{subsection}{Expression}
For any positive integer $r$,  denote by $\Zmod r$ the additive group of integers considered modulo $r$, by $\rous{r}$ the set of primitive roots of unity of order $r$, and by $\rou{r}$ some fixed element of $\rous{r}$.
\begin{lemma}\label{expressionlemma}
Suppose that $\varphi \in \exppoly$ is non-zero and that $\deg \varphi = r$.
Then $r> 0$, and $\varphi_{\lambda} = \varphi_{\zeta \lambda}$ whenever $\lambda, \zeta \in \kcross$ and $\zeta^r = 1$.
Moreover, there exists $\psi \in \exppoly$ such that 
\begin{enumerate}
\item $\varphi = \kronmod r \psi$, and
\item $\charpoly \varphi = \prod_{i \in \Zmod r}{ \charpoly \psi (\rou{r}^i \indett)}$
is a decomposition of $\charpoly \varphi$ into co-prime factors.
\end{enumerate}
\end{lemma}
\begin{proof}
The support of $\varphi$ is contained in $r \Z$.
If $r = 0$, then $\varphi (m) = 0$ for any non-zero $m \in \Z$. 
In particular $\varphi$ has infinitely many consecutive zeros, and so $\varphi = 0$, contrary to hypothesis.
Therefore $r > 0$. 
The support of $\varphi$ is contained in the support of $\kronmod r$, and so
$\varphi = \frac 1 r \kronmod r \varphi$.
Hence
$$ \varphi_\lambda = (\frac 1 r \kronmod r \varphi)_\lambda = \frac 1 r \sum_{i \in \Zmod r} \varphi_{(\rou{r}^i \lambda)},$$
for any $\lambda \in \kcross$.
If $\zeta^r = 1$,
then the expression on the right-hand side is invariant under the substitution $\lambda \mapsto \zeta \lambda$, and so the first claim is proven.

Multiplication by $\rou{r}$ decomposes $\kcross$ into a disjoint union of orbits for the cyclic group $\Zmod r$, and all orbits are of size $r$.
Choose any set $B$ of representatives, so that
$ \kcross = \bigsqcup_{i \in \Zmod r} \rou{r}^i B$.
Then $\psi = \sum_{\lambda \in B} \varphi_\lambda \expmap \lambda$ has the required property, by Proposition \ref{charpolypropn}.
\end{proof}
\begin{remark}
The function $\psi \in \exppoly$ of Lemma \ref{expressionlemma} is not unique.
Indeed, if
$$ \psi = \sum_i a_i \expmap {\mu_i}$$
has the required property, then so does $\psi' = \sum_i a_i \expmap{\rou{r}^{n_i} \mu_i}$ for any $n_i \in \Zmod r$. 
\end{remark}

\end{subsection}
\end{section}

\begin{section}{The Category $\categoryo$}
For any Lie algebra \newnot{notn:loopalgebra}$\mathfrak a$ over $\K$, denote by 
$$\hat{\mathfrak{a}} = \mathfrak{a} \otimes \laurent$$
the \newterm{loop algebra} associated to $\mathfrak{a}$,
with the Lie bracket
\begin{equation}\label{loopalgebrabracket}
\LieBrac{x \otimes \indett^i}{y \otimes \indett^j} = \LieBrac{x}{y} \otimes \indett^{i+j}, \qquad x,y \in \mathfrak{a}, \quad i,j \in \Z.
\end{equation}
Let $\set{\sle, \slf, \slh}$ be a standard basis for $\g = \SL{2}$, so that
$$ \LieBrac{\sle}{\slf} = \slh, \qquad \LieBrac{\slh}{\sle} = 2\sle, \qquad \LieBrac{\slh}{\slf} =-2\slf.$$
The Lie algebra $\g$ has a triangular decomposition
\begin{equation}\label{triangulardecomposition}
\g = \g_{+} \oplus \h \oplus \g_{-}, \qquad \sle \in \g_{+}, \quad \slh \in \h, \quad \slf \in \g_{-}.
\end{equation}
The decomposition (\ref{triangulardecomposition}) defines a decomposition of the loop algebra $\ghat$
$$ \ghat = \ghat_{+} \oplus \hhat \oplus \ghat_{-}$$
as a direct sum of subalgebras.
The centreless affine Lie algebra $\affineg$ is the one-dimensional extension of $\ghat$ by $\K \gd$ defined by
$$ \LieBrac{\gd}{x \otimes \indett^n} = n x \otimes \indett^n, \qquad x \in \g, \quad n \in \Z.$$
Let $\affinecsa = \h \oplus \K \gd$, and let $\groot, \imagroot \in \lineardual \affinecsa$ be given by
$$ \groot (\slh) = 2, \quad \groot(\gd) = 0, \quad \imagroot(\slh) = 0, \quad \imagroot (\gd) = 1.$$
Denote by $\heis$ the Heisenberg subalgebra $\hhat \oplus \K \gd$ of $\affineg$.

For any weight $\affineg$-module $M$, write
$$ \support M = \set{ \chi \in \lineardual \affinecsa | \weightspace M \chi \ne 0}.$$
The \newnot{notn:categoryo}\newterm{category $\categoryo$}, introduced by Chari \cite{Chari86}, consists of the weight $\affineg$-modules $M$ such that
$$ \ts \support M \subset \bigcup_{\lambda \in A} (\lambda - \zplus \groot \times \Z \imagroot) $$
for some finite subset $A = A_M \subset \lineardual \h$.
The morphisms of the category are the homomorphisms of $\affineg$-modules.

\begin{subsection}{The modules $\affineirred \varphi$}\label{affineirredsubsection}
Let $\varphi \in \exppoly$.
It is straightforward to verify that $\laurentabbrev$ is a $\heis$-module via
$$ \slh \otimes \indett^m \cdot \indett^n = \varphi(m) \indett^{m+n}, \qquad \gd \cdot \indett^n = n \indett^n, \qquad m,n \in \Z.$$
Denote by $\hmod \varphi$ the $\heis$-submodule generated by $1 \in \laurentabbrev$.
\begin{lemma}\label{hmodlemma}
For any non-zero $\varphi \in \exppoly$,
\begin{enumerate}
\item\label{hmodlemma1} $\hmod \varphi = \sublaurent r$, where $r = \deg \varphi$;
\item\label{hmodlemma2} $\hmod \varphi$ is irreducible.
\end{enumerate}
\end{lemma}
\begin{proof}
By definition, $\hmod \varphi$ is spanned by monomials
\begin{equation}\label{monomialform} \slh \otimes \indett^{m_1} \cdots \slh \otimes \indett^{m_k} \cdot 1 = \prod_{i=1}^k \varphi(m_i) \indett^{m_1 + \cdots m_k},
\end{equation}
where $k \geqslant 0$ and $m_i \in \Z$.
Therefore $\indett^m \in \hmod \varphi$ precisely when $m = m_1 + \cdots + m_k$ for some $m_i$ in the support of $\varphi$.
That is, $\indett^m \in \hmod \varphi$ if and only if $m \in A$ where $A$ is the monoid generated by the support of $\varphi$.
By Lemma \ref{submonoidlemma}, $A = r \Z$, where $r = \deg \varphi$, and so $\hmod \varphi = \sublaurent r$, proving part \ref{hmodlemma1}.

Now suppose that $v \in \hmod \varphi$ is a non-zero weight vector.
Then $v$ is proportional to $\indett^m$ for some $m \in \Z$.
By part \ref{hmodlemma1}, $\indett^{-m} \in \hmod \varphi$ also, and so there exist integers $m_1, \dots, m_k$ in the support of $\varphi$ such that $\indett^{-m}$ is proportional to a monomial \eqref{monomialform}.
Then
$$ 
 \slh \otimes \indett^{m_1} \cdots \slh \otimes \indett^{m_k} \cdot \indett^m = \prod_{i=1}^k \varphi(m_i) \indett^{0}
$$
is proportional to the generator $1 \in \hmod \varphi$.
Thus $\hmod \varphi$ is generated by any non-zero weight vector $v$, and hence is irreducible.
\end{proof}

For any $\varphi \in \exppoly$, let $\ghat_{+} \cdot \hmod \varphi = 0$.
Denote by $\affineirred \varphi$ the unique irreducible quotient of the induced module
$$ \Ind {\heis + \ghat_{+}} \affineg \hmod \varphi.$$
\end{subsection}

\begin{subsection}{Irreducible modules in $\categoryo$}
For any $\gamma \in \K$, denote by $\onedim \gamma$ the one-dimensional $\affineg$-module given by
$$ \ghat \cdot \onedimvector \gamma = 0, \qquad \gd \cdot \onedimvector \gamma = \gamma \onedimvector \gamma.$$
It is apparent from subsection \ref{affineirredsubsection} that one may in fact construct an irreducible $\affineg$-module $\affineirred \varphi$ for any $\varphi \in \allfunctions$ such that the monoid generated by the support of $\varphi$ has the form $r \Z$, for some non-negative integer $r$.
Chari \cite{Chari86} demonstrates that any irreducible object of the category $\categoryo$ can be obtained as the tensor product of such a module $\affineirred \varphi$ with a one-dimensional module $\onedim \gamma$.
On the other hand, it follows from the work of Billig and Zhao \cite{BilligZhao} that $\affineirred \varphi$ has all weight spaces finite dimensional precisely when $\varphi \in \exppoly$.
Thus:
\begin{theorem}\cite{BilligZhao, Chari86}
Suppose that $M$ is an irreducible module in the category $\categoryo$, and that all weight spaces of $M$ are finite dimensional.  Then $M \cong \affineirred \varphi \bigotimes \onedim \gamma$ for some $\varphi \in \exppoly$ and $\gamma \in \K$.
\end{theorem}
\end{subsection}

\end{section}

\begin{section}{Loop-Module Realisation}\label{loopmodsection}
For any $\ghat$-module $M$, denote by
$$ \loopmod M = M \otimes \laurent$$
the $\affineg$-module defined by
$$ x \otimes \indett^m \cdot v \otimes \indett^n = (x \otimes \indett^m \cdot v) \otimes \indett^{m+n}, \qquad \gd \cdot v \otimes \indett^n = n v \otimes \indett^n,$$
for any $m,n \in \Z$, $v \in M$, and $x \in \g$.
Modules thus constructed are called \newterm{loop modules}.
For $\varphi \in \exppoly$, let $\K \gen{\varphi}$ be the one-dimensional $\hhat$-module defined by
$$ \slh \otimes a \cdot \gen{\varphi} = (a \cdot \varphi)(0)  \gen{\varphi}, \qquad a \in \laurentabbrev.$$
Let $\ghat_{+} \cdot {\gen \varphi} = 0$, let \newnot{notn:ind2}
$$ \ind \varphi = \Ind{\hhat + \ghat_{+}}{\ghat} {\K \gen{\varphi}}$$
denote the induced $\ghat$-module, and let $\irred \varphi$ denote the unique irreducible quotient of $\ind \varphi$.
In this section, it is shown that if $\varphi \in \exppoly$ is non-zero and $r = \deg \varphi$, then $\affineirred \varphi$ is isomorphic to a direct summand of the loop module $\loopmod {\irred \varphi}$, and moreover that this summand may be described in terms of the semi-invariants of an action of the cyclic group $\Zmod r$ on $\irred \varphi$.
The results of this section are due to Chari and Pressley \cite{ChariPressley1986} (see also \cite{ChariGreensteinLevelZero}).

Denote by $\order \eta$ the order of a finite-order automorphism $\eta$.
If $\eta$ is an endomorphism of a vector space $V$, write
$$
\eigenspace V \eta \lambda = \set{ v \in V | \eta (v) = \lambda v }
$$
for the eigenspace of eigenvalue $\lambda$, for any $\lambda \in \K$.

\begin{subsection}{Cyclic group action on $\irred \varphi$}
\begin{lemma}\label{automorphismlemma}
Suppose that $\varphi \in \exppoly$ is non-zero, and that $\zeta \in \kcross$ is such that $\zeta^{r} = 1$, where $r = \deg \varphi$.
Then for all $a \in \laurentabbrev$,
$$ ( a (\zeta \indett) \cdot \varphi) (0) = (a \cdot \varphi) (0).$$
\end{lemma}
\begin{proof}
The support of $\varphi$ is contained in $r \Z$.
Therefore, if $a(\indett) = \sum_i a_i \indett^i$, then
\begin{eqnarray*}
(a (\zeta \indett) \cdot \varphi) (0) = \sum_{i \equiv 0 \pmod r} a_i \hspace{0.15em} \zeta^i \varphi(i)
	= \sum_{i \equiv 0 \pmod r} a_i \hspace{0.15em} \varphi(i) 
	= (a \cdot \varphi) (0). \mqed
\end{eqnarray*}
\end{proof}

\begin{propn}\label{endomorphismpropn}
Suppose that $\varphi \in \exppoly$ is non-zero, and let $r = \deg \varphi$. 
Then there exists an order-$r$ automorphism $\irredcycle = \irredcycle_\varphi$ of the $\h$-module $\irred \varphi$ defined by $\irredcycle (\gen \varphi) = \gen \varphi$ and
$$ \irredcycle (x \otimes a \cdot w) = x \otimes a(\zeta^{-1} \indett) \cdot \irredcycle(w), \qquad x \in \g, \quad a \in \laurentabbrev, \quad w \in \irred \varphi,$$
where $\zeta = \rou r$.
Moreover, $\irredcycle$ decomposes $\irred \varphi$ as a direct sum of eigenspaces
$$ \irred \varphi = \bigoplus_{i \in \Zmod r} \eigenspace {\irred \varphi} {\irredcycle} {\zeta^i}.$$
\end{propn}
\begin{proof}
The rule $\indett \mapsto \zeta^{-1} \indett$ extends to an automorphism of $\laurentabbrev$, which defines an automorphism of the loop algebra $\ghat$. 
This automorphism in turn defines an automorphism $\irredcycle$ of the universal enveloping algebra $\UEA \ghat$.
The universal module $\ind \varphi$ may be realised as the quotient of $\UEA \ghat$ by the left ideal $I$ generated by $\ghat_{+}$ and by the elements of the set
$$
\set{ \slh \otimes a - (a \cdot \varphi) (0)| a \in \laurentabbrev}.
$$
The map $\irredcycle$ preserves this set by Lemma \ref{automorphismlemma}:
\begin{eqnarray*}
\irredcycle ( \slh \otimes a - (a \cdot \varphi)(0)) &=& \slh \otimes a(\zeta^{-1} \indett) - (a \cdot \varphi) (0) \\
	&=& \slh \otimes a(\zeta^{-1} \indett) - (a(\zeta^{-1} \indett) \cdot \varphi) (0).
\end{eqnarray*}
Clearly $\irredcycle$ preserves $\ghat_{+}$, and so $\irredcycle (I) = I$.
Therefore $\irredcycle$ is well-defined on the quotient $\ind \varphi$ of $\UEA{\ghat}$.
The monomial
$$
\slf \otimes \indett^{n_1} \cdots \slf \otimes \indett^{n_k} \cdot \gen \varphi \quad \in \ind \varphi
$$
is an eigenvector of eigenvalue $\zeta^{-m}$ where $m = \sum_{i=1}^k n_i$,
and so 
the Poincar\'e-Birkhoff-Witt Theorem guarantees a decomposition
\begin{equation}\label{indeigenspaces}
\ind \varphi = \bigoplus_{i \in \Zmod r} \eigenspace {\ind \varphi} \irredcycle {\zeta^i}
\end{equation}
of $\ind \varphi$ into eigenspaces for $\irredcycle$.
It is easy to check that $\irredcycle$ commutes with the action of $\slh \otimes \indett^0$, from which it follows that
if $U$ is a proper submodule, then so is $\irredcycle (U)$. 
Hence $\irredcycle$ preserves the maximal submodule of $\ind \varphi$, and so is defined on the quotient $\irred \varphi$. 
This induced map is of order $r$, by construction, and decomposes $\irred \varphi$ in the manner claimed by (\ref{indeigenspaces}).
\end{proof}
\end{subsection}

\begin{subsection}{Decomposition of a loop module}
For any non-zero $\varphi \in \exppoly$, define an automorphism \newnot{notn:loopend}$\loopend_\varphi$ of the vector space $\loopmod {\irred \varphi}$ via
$$ \loopend_\varphi ( u \otimes a ) = \irredcycle_\varphi (u) \otimes a(\rou{r} \indett), \qquad u \in \irred \varphi, \quad a \in \laurentabbrev,$$
where $r = \deg \varphi$.
\begin{theorem}\label{loopendtheorem}
Suppose that $\varphi \in \exppoly$ is non-zero. Let $r = \deg \varphi$,
$\zeta = \rou r$ and $\loopend = \loopend_\varphi$.
Then:
\begin{enumerate}
\item\label{automorphismpart} $\loopend$ is automorphism of the $\affineg$-module $\loopmod {\irred \varphi}$ of order $r$;
\item\label{decompositionpart} $\loopend$ decomposes $\loopmod {\irred \varphi}$ as a direct sum of eigenspaces
$$ \loopmod {\irred \varphi} = \bigoplus_{i \in \Zmod r} \eigenspace {\loopmod {\irred \varphi}} \loopend {\zeta^i},$$
where
$$
\eigenspace {\loopmod {\irred \varphi}} \loopend {\zeta^i} = \bigoplus_{m \in \Z} \eigenspace {\irred \varphi} {\irredcycle_\varphi} {\zeta^{i-m}} \otimes \indett^m, \qquad i \in \Zmod r;
$$
\item\label{realizationpart} 
For any $i \in \Zmod r$, the $\affineg$-modules $\eigenspace {\loopmod {\irred \varphi}} \loopend {\zeta^i}$ and $\affineirred \varphi \otimes \onedim i$ are isomorphic. 
\end{enumerate}
\end{theorem}
\begin{proof}
For any $x \in \g$, $u \in \irred \varphi$, $a \in \laurentabbrev$ and $m \in \Z$, 
\begin{eqnarray*}
\loopend ( {x \otimes \indett^m} \cdot {u \otimes a}) &=& \loopend ( ({x \otimes \indett^m} \cdot u) \otimes \indett^m a)\\
	&=& \zeta^m \irredcycle ({x \otimes \indett^m} \cdot u) \otimes \indett^m a (\zeta \indett) \\
	&=& \zeta^m  \zeta^{-m} ({x \otimes \indett^m} \cdot \irredcycle (u)) \otimes \indett^m a (\zeta \indett) \\
	&=& {x \otimes \indett^m} \cdot (\irredcycle(u) \otimes a(\zeta \indett)) \\
	&=& {x \otimes \indett^m} \cdot \loopend (u \otimes a),
\end{eqnarray*}
where $\irredcycle = \irredcycle_\varphi$. 
The map $\loopend$ is of order $r$ by definition, and so part \ref{automorphismpart} is proven.
Part \ref{decompositionpart} follows immediately from Proposition \ref{endomorphismpropn}.
Let $i \in \Zmod r$, and write 
$$ U = \eigenspace {\loopmod {\irred \varphi}} \loopend {\zeta^i}, \qquad U' =  \bigoplus_{m \in \Z} \weightspace U { \frac{\varphi(0)}{2} \groot + m \imagroot}.$$
Then $U'$ and $\hmod \varphi \otimes \onedim i$ are isomorphic as $\heis$-modules, via
$$ \gen{\varphi} \otimes \indett^{m r + i} \mapsto \indett^{m r} \otimes \onedimvector i , \qquad m \in \Z.$$
This map extends uniquely to an epimorphism of $\affineg$-modules
$U \to \affineirred \varphi \otimes \onedim i$.
Therefore it is sufficient to prove that $U$ is an irreducible $\affineg$-module.
Suppose that $W$ is a submodule of $U$.
Then $W$ contains a non-zero maximal weight vector $v \otimes \indett^n$.
The $\ghat$-module epimorphism $U \to \irred \varphi$ that is induced by $\indett \mapsto 1$ maps this element to 
 a non-zero maximal vector of $\irred \varphi$. 
Therefore $v = \lambda \gen{\varphi}$ is a non-zero scalar multiple of the highest-weight vector.
Hence $W$ has non-trivial intersection with the generating subspace $U'$ of $U$.
The $\heis$-module $U'$ is irreducible, so
$U' \subset W$, and thus $W=U$. 
Therefore $U$ is irreducible.
\end{proof}
\end{subsection}

\begin{subsection}{Characters and semi-invariants}
If $\varphi \in \exppoly$ is non-zero and $r = \deg \varphi$, then
Theorem \ref{loopendtheorem} describes the modules $\affineirred \varphi$ in terms of the semi-invariants of $\irred \varphi$ with respect to the action of the cyclic group $\Zmod r$ defined by $\irredcycle$.
In particular, we have the following description of the character of $\affineirred \varphi$.
\begin{corollary}\label{characterinvariant}
Suppose that $\varphi \in \exppoly$ is non-zero and let $\deg \varphi = r$.  Then
$$ \exppolychar \varphi = \indet{Z}^{\frac{\varphi (0)}{2} \groot} \cdot \sum_{k \geqslant 0} \sum_{n \in \Z} \dim \eigenspace {\weightspace{\irred \varphi} {\frac{\varphi (0)}{2} \groot - k \groot}} \irredcycle {\zeta^n} \hspace{0.15em} \indet{Z}^{- k \groot + n \imagroot},$$
where $\zeta = \rou{r}$.
\end{corollary}
\end{subsection}
\end{section}

\begin{section}{Semi-invariants of Actions of Finite Cyclic Groups}\label{cyclicsection}
A \newterm{$\zplus$-graded vector-space} is a vector space $V$ over $\K$ with a decomposition $V = \bigoplus_{k \geqslant 0} {\vscomponent V k}$ of $V$ into finite-dimensional subspaces  indexed by $\zplus$.
If $V$ is a $\zplus$-graded vector space and $r$ is a positive integer, then the tensor power
$$ V^r := V \otimes \cdots \otimes V \qquad \text{($r$ times)}$$
is also a $\zplus$-graded vector space, with the decomposition 
$$ \textstyle V^r = \bigoplus_{k \geqslant 0} \vscomponent {V^r} k , \qquad \vscomponent {V^r} k = \bigoplus_{k_1 + \cdots + k_r = k} {\vscomponent V {k_1}} \otimes \cdots \otimes {\vscomponent V {k_r}}.$$
The finite cyclic group $\Zmod r$ acts on $V^r$ by cycling homogeneous tensors;
the generator $1 \in \Zmod r$ acts via the vector space automorphism
$$ \cycle r : v_1 \otimes \cdots \otimes v_r \mapsto v_{r} \otimes v_1 \cdots \otimes v_{r-1}, \qquad v_i \in V,$$
and this action preserves the grading, so that $\cycle r (\vscomponent {V^r} k) = \vscomponent {V^r} k$, for any $k \geqslant 0$.
For any $U \subset V^r$, let
$$ U_n = \eigenspace U {\cycle r} {\rou{r}^n}, \qquad n \in \Z.$$
The automorphism $\cycle r$ decomposes $V^r$ as a direct sum of $\zplus$-graded vector spaces 
$$ \textstyle V^r = \bigoplus_{n \in \Zmod r} {V^r_n}, \qquad V^r_n = \bigoplus_{k \geqslant 0} {\vscomponent {V^r_n} k},\qquad {\vscomponent {V^r_n} k} = ({\vscomponent {V^r} k })_n.$$
Associated to any $\zplus$-graded vector space $U$ is the Poincar\'e series \newnot{notn:vscharacter}
$$ \vscharacter U (\indet{X}) = \sum_{k \geqslant 0}{\dim U_k \hspace{0.1em} \indet{X}^k} \quad \in \zplus[[\indet{X}]].$$

\begin{theoremnonum}
For any $\zplus$-graded vector space $V$, $r > 0$ and $n \in \Z$,
$$\vscharacter{ V^r_n} (\indet X) = \frac{1}{r} \sum_{\divides d r } { \ramanujan d n \left ( \vscharacter{V}(\indet{X}^d) \right )^{\frac r d}}.$$
\end{theoremnonum}
In this section, we describe an elementary proof of this statement.
If $U$ is the regular representation of $\Zmod r$ and $V = \SymAlg {U}$ is the symmetric algebra, then the statement follows easily from Molien's Theorem and the identity (\ref{ramanujanidentities}).

Let $\moebiusmu$\newnot{notn:moebiusmu} denote the M\"obius function, \ie the function
$$\moebiusmu : \N \to \set{-1, 0, 1}$$
such that $\moebiusmu(d) = (-1)^l$ if $d$ is the product of $l$ distinct primes, $l \geqslant 0$, and $\moebiusmu(d) = 0$ otherwise.
For any $r > 0$, the function $\moebiusmu$ satisfies the fundamental property
\begin{equation}\label{moebiusproperty}
\sum_{\divides d r}{ \moebiusmu(d)} = \kronecker r 1,
\end{equation}
where $\kroneckersymbol$ denotes the Kronecker function.
A summation $\sum_{\divides d r} {a_d}$ is to be understood as the sum of all the $a_d$ where $d$ is a positive divisor of $r$.
Let $\eulertotient: \N \to \N$\newnot{notn:eulertotient} denote Euler's totient function, so that 
$$
\eulertotient (d) = \# \set{ 0 < k \leqslant d | \gcd (k,d) = 1}, \qquad d > 0.
$$
For any positive integer $d$ and $n \in \Z$, the quantity $\ramanujan d n$ defined by \eqref{ramanujandefn} is called a \newterm{Ramanujan sum}, a \newterm{von Sterneck function}, or a \newterm{modified Euler number}.
These quantities have extensive applications in number theory (see, for example, \cite{Ramanujan}, \cite{NicolVanDiver}), although we require only the most basic properties, such as those described in  \cite{Erdelyi}.
In particular, we note the identities
\begin{equation}\label{ramanujanidentities}
\ramanujan r n = \sum_{\zeta \in \rous r} \zeta^n = \sum_{\divides d {\gcd (r,n)}} d \hspace{0.1em} \moebiusmu(\frac r d) 
\end{equation}
where $n \in \Z$ and $r > 0$.

Fix a $\zplus$-graded vector space $V$, 
let
$$ \basis = \set{ (k,s) \in \zplus^2 | 1 \leqslant s \leqslant \dim \vscomponent V k} $$
and for each $k \geqslant 0$, choose a basis $\set{ v^k_{s} }_{1 \leqslant s \leqslant \dim \vscomponent V k}$ for $\vscomponent V k$.
For any $r > 0$ and $k \geqslant 0$, let
$$ \params r k  = \set { ( (k_1, s_1), \dots, (k_r, s_r)) \in {\basis} ^r | {\textstyle \sum_{i=1}^r k_i = k }}.$$
The elements of $\params r k$ parameterise a graded basis of $\vscomponent {V^r} k$:
$$ \params r k \ni \quad ( (k_1, s_1), \dots, (k_r, s_r)) = I \quad \leftrightarrow \quad v_I = v^{k_1}_{s_1} \otimes \cdots \otimes v^{k_r}_{s_r} \quad \in \vscomponent V k.$$
Define an automorphism $\paramcycle r$ of the sets $\params r k$ via the rule
$$ \paramcycle r : ( (k_1, s_1), \dots, (k_r, s_r)) \mapsto( (k_r, s_r), (k_1, s_1), \dots, (k_{r-1}, s_{r-1}) ).$$
The automorphisms $\cycle r $ and $\paramcycle r$ are compatible in the sense that
$$ \cycle r (v_I) = v_{\paramcycle r (I)}, \qquad I \in \params r k, \quad k \geqslant 0.$$
For $I \in \params r k$, write $\order I = d$ for the minimal positive integer such that $({\paramcycle r})^d (I) = I$.
For any positive divisor $d$ of $r$, let
$$ \ordercount r d k = \# \set{I \in \params r k | \order I = d}, \qquad k \geqslant 0,$$
and write $\ordergf r d (\indet{X}) = \sum_{k \geqslant 0} \ordercount r d k \hspace{0.1em} \indet{X}^k$ for the generating function.
It is apparent that
\begin{equation}\label{orderequation1}
\vscharacter {V^r} (\indet{X}) = (\vscharacter {V} (\indet{X}))^r = \sum_{\divides d r} \ordergf r d (\indet{X})
\end{equation}
\begin{lemma}\label{divisorlemma}
Suppose that $l,r$ are positive integers and that $\divides l r$. Then
$$ \set{ d | d > 0, \ \ \divides {r/l} d \ \ \text{and} \ \ \divides d r } 
= \set{ {r}/{d'} | d' > 0, \ \ \divides {d'} l }.$$
\end{lemma}
\begin{proof}
If $d > 0$ and $\divides {\frac r l} d$, then there exists some positive integer $s$ such that
$$ d = \frac{r}{l} s = \frac r {(l/s)};$$
if in addition $\divides d r$, then $d' := l/s$ is a positive integer, and so $d = r / {d'}$ with $\divides {d'} l$.
Conversely, if  $\divides {d'} l$, then $\divides {r/l} {r/{d'}}$, and it is obvious that $\divides {r/{d'}} r$. 
\end{proof}

\begin{propn}\label{ordercountingpropn1}
For any $\zplus$-graded vector space $V$ and any $r > 0$,
$$ \vscharacter{V^r_n} (\indet{X})= \frac{1}{r} \sum_{ \divides d \hspace{0.1em} {\gcd (r,n)} } d \ordergf r {\frac{r}{d}} (\indet{X}).$$
for all $n \in \Z$.
\end{propn}
\begin{proof}
Suppose that $k \geqslant 0$,
and write $\params r k = \bigsqcup_{O \in \mathcal{P}} O$ for the decomposition of $\params r k$ into a disjoint union of orbits for the action of $\Zmod r$ defined by $\paramcycle r$.
Then
$$ \vscomponent {V^r} k = \bigoplus_{O \in \mathcal{P}} U_O, \qquad U_O = \Span \set{ v_I | I \in O},$$
and moreover $\cycle r (U_O) = U_O$.
For any orbit $O \in \mathcal{P}$, the action of $\cycle r$ on $U_O$ defines the regular representation of $\Zmod d$, where $d = \# O$ is the size of the orbit; in particular, the eigenvalues of $\cycle r$ on $U_O$ are precisely the roots of unity $\zeta$ such that $\zeta^d = 1$, each with multiplicity 1.
Now ${\rou{r}^n}$ is of order $\frac{r}{\gcd(r,n)}$.
Therefore,
\begin{eqnarray*}
\dim \vscomponent {V^r_n} k &=& \# \set{ O \in \mathcal{P} | \divides{ \SFrac{r}{\gcd (r,n)}} {\# O}} \\
	&=& \# \set{ O \in \mathcal{P} | \# O = \SFrac{r}{d} \ \ \text{for some} \ \ \divides d {\gcd (r,n)} }
\end{eqnarray*}
where the last equality follows from Lemma \ref{divisorlemma} with $l = \gcd (r,n)$.
The number of orbits $O \in \mathcal{P}$ of size $r/d$ is precisely $d/r \cdot \ordercount r {r/d} k$.
It follows therefore that
$$ \dim \vscomponent {V^r_n} k = \sum_{\divides d {\gcd (r,n)}} {\SFrac d r \ordercount r {\frac{r}{d}} k },$$
which yields the required equality of generating functions.
\end{proof}

\begin{propn}\label{orderreductionpropn}
For any $\zplus$-graded vector space $V$ and positive integers $r,d$ with $\divides d r$, 
$$
\ordergf r d (\indet{X}) = \ordergf d d (\indet{X}^{\frac r d})
$$
\end{propn}
\begin{proof}
Suppose that $k \geqslant 0$, that
$$ I = ( (k_1, s_1), \dots, (k_r, s_r) ) \in \params r k$$
and that $\order I = d$.
Then 
$$ I' = ( (k_1, s_1), \dots, (k_d, s_d) ) \in \params d {\frac{kd}{r}}.$$
and $\order {I'} = d$.
This establishes a bijection between order-$d$ elements of the sets $\params r k $ and $\params d {\frac{kd}{r}}$, and so $\ordercount r d k = \ordercount d d {\frac{kd}{r}}$.
Therefore
\begin{eqnarray*}
\ordergf r d (\indet X) &=& \sum_{k \geqslant 0}{\ordercount d d {\SFrac{kd}{r}} \hspace{0.1em} \indet{X}^k } \\
&=& \sum_{k \geqslant 0}{\ordercount d d k \hspace{0.1em} (\indet{X}^{\frac r d})^k }
\\
&=& \ordergf d d (\indet{X}^{\frac r d}).\mqed
\end{eqnarray*}
\end{proof}

It follows immediately from Proposition \ref{orderreductionpropn} and equation (\ref{orderequation1}) that
\begin{equation}\label{orderequation2}
(\vscharacter V (\indet{X}))^r = \sum_{\divides d r} {\ordergf d d (\indet{X}^{\frac r d})}.
\end{equation}

\begin{propn}\label{orderformula}
For any $\zplus$-graded vector space $V$ and $r > 0$,
$$
\ordergf r r (\indet{X}) = \sum_{\divides d r} \moebiusmu (d) \left ( \vscharacter V (\indet{X}^d) \right )^{\frac r d}.
$$
\end{propn}
\begin{proof}
The claim is trivial if $r=1$, so suppose that $s > 1$ and that the claim holds for all $0 < r < s$.
Then:
\begin{eqnarray*}
\ordergf s s (\indet X)
&=& \left(\vscharacter V (\indet X)\right)^s- \sum_{\divides d s, d \ne s} \ordergf d d (\indet{X}^{\frac r d}) \ENote{by equation (\ref{orderequation2})} \\
	&=& \left(\vscharacter V (\indet X)\right)^s- \sum_{\divides d s, d \ne s} \sum_{\divides {d'} d} \moebiusmu(d') \left ( \vscharacter V (\indet{X}^{\frac{s d'} d}) \right )^{\frac d {d'}} \ENote{by inductive hypothesis} \\
	&=& \left(\vscharacter V (\indet X)\right)^s- \sum_{\divides e s, e \ne 1} \hspace{0.1em} \big( \sum_{\divides d e, d \ne e} \moebiusmu(d) \big )\hspace{0.1em} \left ( \vscharacter V (\indet{X}^e) \right )^{\frac s e}\\
	&& \Note{write $e = \SFrac{s d'}{d}$ and use Lemma \ref{divisorlemma}} \\
	&=& \left(\vscharacter V (\indet X)\right)^s- \sum_{\divides e s, e \ne 1} ( - \moebiusmu(e) ) \left ( \vscharacter V (\indet{X}^e) \right )^{\frac s e} \ENote{by equation (\ref{moebiusproperty})}\\
	&=&  \sum_{\divides e s} \moebiusmu (e) \left ( \vscharacter V (\indet{X}^e) \right )^{\frac s e},
\end{eqnarray*}
and so the claim holds for $s$ also.
\end{proof}

\begin{theorem}\label{semiinvarianttheorem}
For any $\zplus$-graded vector space $V$, $r > 0$ and $n \in \Z$,
$$\vscharacter{ V^r_n} (\indet X) = \frac{1}{r} \sum_{\divides d r } { \ramanujan d n \left ( \vscharacter{V}(\indet{X}^d) \right )^{\frac r d}}.$$
\end{theorem}
\begin{proof}
For any $n \in \Z$,
\begin{eqnarray*}
\vscharacter {V^r_n} (\indet X) &=& \frac 1 r \sum_{\divides d {\gcd (r,n)}} d \hspace{0.1em}\ordergf r {\frac r d} (\indet X) \ENote{by Proposition \ref{ordercountingpropn1}} \\
&=& \frac 1 r \sum_{\divides d {\gcd (r,n)}} d \hspace{0.1em}\ordergf {\frac r d} {\frac r d} ( \indet{X}^d) \ENote{by Proposition \ref{orderreductionpropn}}\\
&=& \frac 1 r \sum_{\divides d {\gcd (r,n)}} d \sum_{\divides {d'} {\frac r d}} \moebiusmu(d') \left( \vscharacter V (\indet{X}^{dd'}) \right)^{\frac r {d d'}}  \ENote{by Proposition \ref{orderformula}} \\
&=& \frac 1 r \sum_{\divides e r} \big ( \sum_{\divides d {\gcd (e,n)}} {d \hspace{0.1em} \moebiusmu(\frac e d)} \big ) \left ( \vscharacter V (\indet{X}^e) \right )^{\frac r e} \\
&=& \frac 1 r \sum_{\divides e r}{ \ramanujan e n \left ( \vscharacter V (\indet{X}^e) \right )^{\frac r e}},
\end{eqnarray*}
where the last equality follows from equation (\ref{ramanujanidentities}).
\end{proof}
\end{section}

\begin{section}{Character Formulae}
In this section, we show that if $\varphi \in \exppoly$, then the module $\irred \varphi$ is an irreducible highest-weight module for the truncated current Lie algebra $\epquotient \g \varphi$.
An explicit formula for the character of such a module was obtained in \cite{WilsonTCLA}.
Therefore, we are able to derive an explicit formula for $\exppolychar \varphi$ by employing the results of Sections \ref{loopmodsection} and \ref{cyclicsection}.

\begin{subsection}{Modules for truncated current Lie algebras}
Associated to any Lie algebra $\mathfrak a$  and non-zero $\varphi \in \exppoly$ is the \newterm{truncated current Lie algebra} $\epquotient {\mathfrak a} \varphi$,  
$$\epquotient {\mathfrak a} \varphi = \mathfrak{a} \otimes \frac{\laurentabbrev}{\charpoly \varphi \laurentabbrev},$$ 
with the Lie bracket given by equation \eqref{loopalgebrabracket}.

\begin{propn}\label{charpolyactstrivially}
Suppose that $\varphi \in \exppoly$ is non-zero.
Then the defining ideal $\g \otimes \charpoly \varphi \laurentabbrev \subset \ghat$ acts trivially on the $\ghat$-module $\irred \varphi$, and so $\irred \varphi$ is a $\epquotient \g \varphi$-module.
\end{propn}
\begin{proof}
Let $\K \gen{+}$ denote the one-dimensional $\hhat$-module defined by 
$$ \slh \otimes a \cdot \gen{+} = (a \cdot \varphi) (0) \gen{+}, \qquad a \in \laurentabbrev.$$
Then by definition of the characteristic polynomial $\charpoly \varphi$, the subalgebra $\h \otimes \charpoly \varphi \laurentabbrev \subset \hhat$ acts trivially upon $\gen{+}$, and so $\K \gen{+}$ may be considered as an $\epquotient \h \varphi$-module.
Let $\epquotient {\g_{+}} \varphi \cdot \gen{+} = 0$, and let 
$$ M = \Ind{\epquotient \h \varphi + \epquotient {\g_{+}} \varphi}{\epquotient \g \varphi} {\K \gen{+}} $$
denote the induced $\epquotient \g \varphi$-module.
Denote by $L$ the unique irreducible quotient of $M$.
Then $L$ is a $\ghat$-module, via the canonical epimorphism
$\ghat \epiarrow \epquotient \g \varphi$,
and is irreducible with highest-weight defined by the function $\varphi$.
Hence $\irred \varphi \cong L$ as $\ghat$-modules, and the claim follows from the construction of $L$.
\end{proof}
\end{subsection}

\begin{subsection}{Tensor products}
\begin{propn}\label{TensorSimplePropn}
Let $\varphi_1, \varphi_2 \in \exppoly$.  Then
$$\textstyle \irred { \varphi_1 + \varphi_2 } \cong \irred {\varphi_1} \bigotimes \irred {\varphi_2},$$
as $\ghat$-modules if $\charpoly {\varphi_1}$ and $\charpoly {\varphi_2}$ are co-prime.
\end{propn}
\begin{proof}
Let $\varphi = \varphi_1 + \varphi_2$.
Then $\charpoly \varphi = \charpoly {\varphi_1} \charpoly {\varphi_2}$ since $\charpoly {\varphi_1}$ and $\charpoly {\varphi_2}$ are co-prime.
By Proposition \ref{charpolyactstrivially}, $\irred \varphi$ is an irreducible module for $\epquotient \g \varphi$, and by the Chinese Remainder Theorem,
\begin{equation}\label{LADecomp}
\epquotient \g \varphi \cong \epquotient \g {\varphi_1} \oplus \epquotient \g {\varphi_2}.
\end{equation}
By Proposition \ref{charpolyactstrivially}, $\irred {\varphi_i}$ is a module for $\epquotient \g {\varphi_i} $, $i = 1,2$.
The Lie algebra $\epquotient \g {\varphi_i}$ is finite-dimensional, and $\K$ is algebraically closed, and so $\UEA{\epquotient \g {\varphi_i}}$ is Schurian \cite{Quillen1969}, $i=1,2$.
Thus $\UEA{\epquotient \g {\varphi_i}}$ is tensor-simple \cite{Bavula1995}, and so 
$ \irred {\varphi_1} \bigotimes \irred {\varphi_2} $
is an irreducible module for $\UEA{\epquotient \g {\varphi_1} } \bigotimes \UEA{\epquotient \g {\varphi_2}}$.
The decomposition \eqref{LADecomp} and the Poincar\'e-Birkhoff-Witt Theorem imply that
$$ \textstyle \UEA{\epquotient \g {\varphi_1}} \bigotimes \UEA{\epquotient \g {\varphi_2}} \cong \UEA{\epquotient \g \varphi},$$
and so $ \irred {\varphi_1} \bigotimes \irred {\varphi_2} $ is an irreducible module for $\epquotient \g \varphi$.
The irreducible highest-weight modules $\irred \varphi$ and $\irred {\varphi_1} \bigotimes \irred {\varphi_2}$ are of equal highest weight, by the Leibniz rule, and hence are isomorphic.
\end{proof}
\end{subsection}

\begin{subsection}{Semi-invariants of the modules $\irred \varphi$}
For any $\varphi \in \exppoly$, consider $\irred \varphi$ as a {$\zplus$-graded} vector space via
$$ \irred \varphi = \bigoplus_{k \geqslant 0}{\vscomponent {\irred \varphi} k}, \qquad \vscomponent {\irred \varphi} k = \weightspace {\irred \varphi} {\varphi (0) \frac{\groot}{2} - k \groot}\ , \quad k \geqslant 0.$$

\begin{propn}\label{recognitionofcyclicend}
Suppose that $\varphi \in \exppoly$ is non-zero and that $\varphi = \kronmod r \psi$, where $r = \deg \varphi$ and $\psi \in \exppoly$, as per Lemma \ref{expressionlemma}.
Then there exists an isomorphism
$$
\Omega : \irred \varphi \to {\irred \psi}^r
$$
of $\zplus$-graded vector spaces such that $\cycle r = \Omega \circ \irredcycle_\varphi \circ \Omega^{-1}$.
\end{propn}
\begin{proof}
For $j \in \Zmod r$, write $\psi^j = \expmap{\rou{r}^{-j}} \psi$. 
Then $\charpoly \varphi = \prod_{j \in \Zmod r} \charpoly{\psi^j}$ is a decomposition of $\charpoly \varphi$ into co-prime factors, and $\varphi = \sum_{j \in \Zmod r}{ \psi^j}$.
By the Chinese Remainder Theorem, there exists a finite linearly independent set $\set{a_i | i \in I} \subset \laurentabbrev$ such that $\set{a_i + \charpoly{\psi} \laurentabbrev | i \in I}$ is a basis for $\laurentabbrev / \charpoly {\psi} \laurentabbrev$ and
$$ a_i \equiv 0 \pmod {\charpoly {\psi^j}}, \quad j \not \equiv 0 \pmod r, \quad i \in I, \quad j \in \Zmod r.$$
Write $a_{i,j}(\indett) = a_i (\rou{r}^{-j} \indett)$, $i \in I$, $j \in \Zmod r$.
Then by symmetry, $\set{a_{i,j} + \charpoly{\psi^j} \laurentabbrev | i \in I}$ is a basis for $\laurentabbrev / \charpoly{\psi^j} \laurentabbrev$ and 
$$
a_{i,j} \equiv 0 \pmod {\charpoly {\psi^k}}, \quad j \not \equiv k \pmod r, \quad i \in I, \quad j,k \in \Zmod r.
$$
For any $i \in I$ and $j \in \Zmod r$,
\begin{equation}\label{howitcycles}
\irredcycle_\varphi ( \slf \otimes a_{i,j} \cdot w) = \slf \otimes {a_{i,j} (\rou{r}^{-1} \indett)} \cdot \irredcycle_\varphi (w) = \slf \otimes a_{i, j+1} \cdot \irredcycle_\varphi (w), \qquad w \in \irred \varphi.
\end{equation}
By Proposition \ref{TensorSimplePropn}, there exists an isomorphism
$$
\ts \Upsilon : \irred \varphi \to \bigotimes_{j \in \Zmod r} \irred {\psi^j}
$$
of $\ghat$-modules, and we may assume that $\Upsilon (\gen \varphi) = \otimes_{j \in \Zmod r} \gen {\psi^j}$.
For any $k \in \Zmod r$, identify 
\begin{equation}\label{identification}
\ts \irred {\psi^k} = 1 \otimes \cdots \otimes \irred {\psi^k} \otimes \cdots \otimes 1\subset \bigotimes_{j \in \Zmod r}{\irred {\psi^j}}.
\end{equation}
Then $\irred {\psi^k}$ is generated by the action of the basis $\set{ \slf \otimes a_{i,k} | i \in I}$ of $\epquotient {\g_{-}} {\psi^k}$ on the highest-weight vector $\Upsilon(\gen \varphi)$. 
Therefore, modulo the identification (\ref{identification}), 
$$(\Upsilon \circ \irredcycle_{\varphi} \circ \Upsilon^{-1})(\irred {\psi^k}) \subset \irred {\psi^{k+1}}, \qquad k \in \Zmod r,$$
by equation (\ref{howitcycles}).
Since $\irredcycle_\varphi$ is an automorphism of the $\zplus$-graded vector space $\irred \varphi$, the restriction 
$$
(\Upsilon \circ \irredcycle_{\varphi} \circ \Upsilon^{-1}) :  \irred {\psi^k} \to \irred {\psi^{k+1}}, 
$$
is an isomorphism of the $\zplus$-graded vector spaces.
These isomorphisms obviously induce isomorphisms $\epsilon_j : \irred {\psi^j} \to \irred {\psi^0} = \irred \psi$, and
$$
\textstyle \epsilon_j : \prod_{i \in I} (\slf \otimes a_{i,j})^{k_i} \cdot \gen {\psi^j} \mapsto \prod_{i \in I} (\slf \otimes a_{i})^{k_i} \cdot \gen \psi,
$$
by equation (\ref{howitcycles}).
Let $\epsilon = \bigotimes_{j \in \Zmod r} \epsilon_j$, and write $\Omega$ for the composition
$$ \epsilon \circ \Upsilon : \irred \varphi \to \irred{\psi}^r.$$
The vector space $\irred{\psi}^r$ is spanned by the homogeneous tensors
$$
\textstyle \bigotimes_{j \in \Zmod r} \prod_{i \in I} (\slf \otimes a_i)^{k_{i,j}}  \gen \psi, \qquad k_{i,j} \geqslant 0.
$$
For any homogeneous tensor of this form
\begin{eqnarray*}
(\Omega \circ \irredcycle_\varphi \circ \Omega^{-1})& \cdot& \textstyle (\bigotimes_{j \in \Zmod r} \prod_{i \in I} (\slf \otimes a_i)^{k_{i,j}} \gen \psi) \\
&&= \textstyle \epsilon \circ (\Upsilon \circ \irredcycle_\varphi \circ \Upsilon^{-1}) (\bigotimes_{j \in \Zmod r} \prod_{i \in I} (\slf \otimes a_{i,j})^{k_{i,j}} \gen {\psi^j}) \\ 
&&= \textstyle \epsilon \circ (\Upsilon \circ \irredcycle_\varphi \circ \Upsilon^{-1}) (\prod_{j \in \Zmod r} \prod_{i \in I} (\slf \otimes a_{i,j})^{k_{i,j}} \cdot \otimes_{j \in \Zmod r} \gen {\psi^j}) \\ 
&&= \textstyle \epsilon (\prod_{j \in \Zmod r} \prod_{i \in I} (\slf \otimes a_{i,j})^{k_{i,j-1}} \cdot \otimes_{j \in \Zmod r} \gen {\psi^j}) \\ 
&&= \textstyle \epsilon (\bigotimes_{j \in \Zmod r} \prod_{i \in I} (\slf \otimes a_{i,j})^{k_{i,j-1}} \gen {\psi^j}) \\ 
&&= \textstyle \bigotimes_{j \in \Zmod r} \prod_{i \in I} (\slf \otimes a_i)^{k_{i,j-1}} \gen \psi \\
&&= \textstyle \cycle r (\bigotimes_{j \in \Zmod r} \prod_{i \in I} (\slf \otimes a_i)^{k_{i,j}} \gen \psi),
\end{eqnarray*}
where the second and fourth equalities are by construction of the polynomials $a_{i,j}$ and the Leibniz rule.
Therefore $\Omega \circ \irredcycle_\varphi \circ \Omega^{-1} = \cycle r$ as required.
\end{proof}
\end{subsection}

\begin{subsection}{Character Formulae}
\begin{theorem}\label{tclacharacter}
Suppose that $a \in \allfunctions$ is a polynomial function and that $\varphi = a \expmap \lambda$ for some $\lambda \in \kcross$.
Then
\begin{equation*}
\vscharacter {\irred \varphi} (\indet X) = 
\begin{cases}
\frac{1 - \indet{X}^{a + 1}}{1 - \indet X} &  \text{if $a \in \zplus$,}\\
\frac 1 {1- \indet X} & \text{otherwise.}
\end{cases}
\end{equation*}
\end{theorem}
\begin{proof}
Let $\Nilp = \deg a$, and write
$ \varphi = \sum_{k=0}^\Nilp a_k \epbasis \lambda k$.
By Proposition \ref{charpolyactstrivially}, $\irred \varphi$ is a module for the truncated current Lie algebra $\epquotient \g \varphi$.
The Cartan subalgebra of $\epquotient \g \varphi$ has a basis
$$ \set { \slh \otimes (\indett - \lambda)^k | 0 \leqslant k \leqslant \Nilp}.$$
By Lemma \ref{annihilationlemma},
$\slh \otimes (\indett - \lambda)^\Nilp$ acts on the highest-weight vector $\gen \varphi$ by the scalar 
\begin{equation}\label{topvalue}
((\indett - \lambda)^\Nilp \cdot \varphi) (0) = \Nilp ! \lambda^\Nilp a_\Nilp.
\end{equation}
If $\Nilp = 0$, then \eqref{topvalue} takes the value $a \in \K$, and so $\irred \varphi$ is the irreducible $\g$-module of highest weight $a$.  Therefore
\begin{equation*}
\vscharacter {\irred \varphi} (\indet X) = 
\begin{cases}
\frac{1 - \indet{X}^{a + 1}}{1 - \indet X} &  \text{if $a \in \zplus$,}\\
\frac 1 {1- \indet X} & \text{otherwise.}
\end{cases}
\end{equation*}
If $\Nilp > 0$, then $a_\Nilp$ is non-zero; thus \eqref{topvalue} is non-zero and the claim follows from Proposition A.1 of \cite{WilsonTCLA}.
\end{proof}

Suppose that $\varphi \in \exppoly$ is non-zero, $\deg \varphi = r$, and that $\psi \in \exppoly$ is given by Lemma \ref{expressionlemma}.
Then 
\begin{equation}\label{niceform}
\psi = \sum_i a_i \expmap {\lambda_i}
\end{equation}
for some finite collection of polynomial functions $a_i \in \allfunctions$ and distinct $\lambda_i \in \kcross$, such that if $(\lambda_i / \lambda_j)^r = 1$, then $i=j$.
\begin{theorem}
Suppose that $\varphi \in \exppoly$ is non-zero, $\deg \varphi = r$, and that
$$ \varphi = \kronmod r \sum_i a_i \expmap {\lambda_i}$$
where the $a_i \in \allfunctions$ and $\lambda_i \in \kcross$ are given by (\ref{niceform}).  Let
$$ \fouriercoeff \varphi (\indet Z)= \frac {\prod_{a_i \in \zplus} (1 - \indet{Z}^{a_i +1})}{(1 - \indet{Z})^M },$$
where $M = \sum_i (\deg a_i + 1)$ and the product is over those indices $i$ such that $a_i \in \zplus$.
Then
$$
\exppolychar \varphi = \indet{Z}^{\varphi (0) \frac{\groot}{2}} \cdot \frac{1}{r} \sum_{n \in \Z} \sum_{\divides d r} \ramanujan d n \left ( \fouriercoeff \varphi (\indet{Z}^{-d\groot}) \right ) ^{\frac r d} \indet{Z}^{n\imagroot}.
$$
\end{theorem}
\begin{proof}
By Corollary \ref{characterinvariant} and Proposition \ref{recognitionofcyclicend},
$$
\exppolychar \varphi = \indet{Z}^{\varphi (0) \frac{\groot}{2}} \cdot \sum_{n \in \Z} \vscharacter {{\irred \psi}^r_n} (\indet{Z}^{-\groot}) \hspace{0.1em} \indet{Z}^{n\imagroot},
$$
and by Theorem \ref{semiinvarianttheorem},
\begin{equation}\label{recursiveformula}
\vscharacter {{\irred \psi}^r_n} (\indet X) = \frac 1 r \sum_{\divides d r} \ramanujan d n \left ( \vscharacter {\irred \psi} (\indet{X}^d) \right )^{\frac r d}.
\end{equation}
By Proposition \ref{TensorSimplePropn}, there is an isomorphism of $\ghat$-modules
$$ \ts \irred \psi \cong \bigotimes_i \irred {\psi^i}, \qquad \psi^i = a_i \expmap {\lambda_i}, $$
since the $\lambda_i$ are distinct.
In particular, $\vscharacter {\irred \psi} = \prod_i \vscharacter {\irred {\psi^i}}$, and so 
$$\vscharacter {\irred \psi} (\indet X)= \fouriercoeff \varphi (\indet X)$$
by Theorem \ref{tclacharacter}.
Therefore the claim follows from equation (\ref{recursiveformula}).
\end{proof}
\end{subsection}
\end{section}

\begin{section}{Acknowledgements}
This work was completed as part of a cotutelle Ph.D. programme under the supervision of Vyacheslav Futorny at the Universidade de S\~ao Paulo and Alexander Molev at the University of Sydney.
The author would like to thank Yuly Billig, Jacob Greenstein, Ivan Dimitrov, Neil Saunders, Tegan Morrison and Senapati Rao for their helpful discussions.
\end{section}

\bibliographystyle{abbrv}
\bibliography{thesis} 

\end{document}